\documentclass[12pt]{amsart}
\usepackage{amssymb,amsmath,amsthm,amssymb}

\usepackage{pgf,tikz}
\usepackage{mathrsfs}
\usetikzlibrary{arrows.meta, intersections}

\usepackage{graphicx}
\usepackage{comment}
\usepackage{hyperref}
\usepackage{xcolor}
\usepackage{epstopdf}
\usepackage{marginnote}
\renewcommand{\epsilon}{\varepsilon}

\def\R{\mathbb{R}}

\def\a{\alpha}
\def\e{\epsilon}
\def\t{\tau}
\def\x{\xi}

\def\th{\theta}

\def\k{\kappa}

\def\wt{\widetilde}

\newtheorem{theorem}{Theorem}[section]
\newtheorem{lemma}[theorem]{Lemma}
\newtheorem{proposition}[theorem]{Proposition}

\newtheorem{corollary}[theorem]{Corollary}

\newtheoremstyle{TheoremNum}
        {\topsep}{\topsep}              
        {\itshape}                      
        {}                              
        {\bfseries}                     
        {.}                             
        { }                             
        {\thmname{#1}\thmnote{ \bfseries #3}}
    \theoremstyle{TheoremNum}

\title[Ancient solutions for the $\k^\a$ flow in $\mathbb R^2$]{Ancient solutions for flow by powers of the curvature in $\mathbb R^2$}

\author[Bourni]{Theodora~Bourni}
\address{Department of Mathematics\\ University of Tennessee\\1403 Circle Dr, Knoxville, TN 37916 USA}
\email{tbourni@utk.edu}

\author[Clutterbuck]{Julie~Clutterbuck}
\address{School of Mathematics \\
Monash University \\
9 Rainforest Walk,
VIC 3800 Australia}
\email{Julie.Clutterbuck@monash.edu}

\author[Nguyen]{Xuan~Hien~Nguyen}
\address{Department of Mathematics \\ Iowa State University \\ 411 Morrill Rd, Ames, IA 50011 USA}
\email{xhnguyen@iastate.edu}

\author[Stancu]{Alina~Stancu}
\address{Department of Mathematics and Statistics \\ Concordia University\\ 1455 Blvd. de Maisonneuve Ouest \\ Montreal, QC, H3G 1M8 Canada}
\email{alina.stancu@concordia.ca}

\author[Wei]{Guofang~Wei}
\address{Department of Mathematics \\ UC  Santa Barbara\\ Santa Barbara, CA 93106 USA}
\email{wei@math.ucsb.edu}

\author[Wheeler]{Valentina-Mira~Wheeler}
\address{Valentina-Mira Wheeler \\ Institute for Mathematics and its Applications \\ University of Wollongong\\ Northfields Avenue\\ Wollongong, NSW 2522 Australia}
\email{vwheeler@uow.edu.au}

\thanks{The research of Julie Clutterbuck was supported by grant FT1301013 of the Australian Research Council. The research of Xuan Hien Nguyen was supported by grant 579756 of the Simons Foundation. The research of Alina Stancu was supported by NSERC Discovery Grant  RGPIN 327635. The research of Guofang Wei was supported by NSF Grant DMS 1811558. The research of Valentina-Mira Wheeler was supported by grants DP180100431 and DE190100379 of the Australian Research Council.}

\begin{document}
\maketitle

\begin{abstract}
 We construct a new compact convex embedded ancient solution of the $\k^\a$ flow in $\R^2$, $\a\in(\frac12,1)$ that lies between two parallel lines. 
Using this solution we classify all convex ancient solutions of the $\k^\a$ flow in $\R^2$, for $\a\in(\frac23,1)$. Moreover, we show that any non-compact convex embedded ancient solution of the $\k^\a$ flow in $\R^2$, $\a\in(\frac12,1)$ must be a translating solution. \end{abstract}

\section{introduction}

A smooth one-parameter family $\{\Gamma_t\}_{t\in I}$ of connected, immersed, planar curves $\Gamma_t\subset \R^2$ \emph{evolves by the $\k^\a$ flow}, $\a>0$, if 
\begin{equation}\label{ka-flow}
\partial_t\gamma(\theta,t)=-\kappa^\a(\theta,t)\nu(\theta,t)\;\;\text{for each}\;\; (\theta,t)\in \Theta\times I
\end{equation}
for some smooth family $\gamma:\Theta\times I\to\R^2$ of immersions of $\Gamma_t$, where $\k(\theta, t)$ and $\nu(\cdot,t)$ are the curvature and the unit normal vector of $\gamma(\cdot,t)$. Our sign convention is that $\vec \kappa=-\kappa\nu$ is the curvature vector. 

We refer to a solution as \emph{compact} if $\Theta\cong S^1$ and \emph{convex} if each of the timeslices $\Gamma_t$ bounds a convex domain, in which case the immersions $\gamma(\cdot,t)$ are proper embeddings.
Both compactness and convexity are properties that are preserved under the flow and it is known that, if the initial curve $\Gamma_0$ is compact and convex, the family converges to a single point in finite time \cite{BA98}. Moreover, if the initial surface is convex then it will immediately become strictly convex and smooth \cite{BA98}. 
The solution $\{\Gamma_t\}_{t\in I}$ is called \emph{ancient} if $I$ contains the interval $(-\infty,t_0)$ for some $t_0\in\R$. In the compact case, by a time translation, we will assume that $I=(-\infty, 0)$. The goal of this paper is to construct and study convex ancient solutions for $\a\in (1/2, 1]$ as well as provide certain classification results.

When $\a=1$, the flow is the famous \emph{curve shortening flow} and such a classification is already known. Daskalopoulos, Hamilton and \v Se\v sum showed that the shrinking circles and the Angenent ovals 
 are the only \emph{compact} examples \cite{DHS}. Their arguments are based on the analysis of a certain Lyapunov functional. Recently, \cite{BLT3}, a new proof of this result was given which  removes the compactness hypothesis by adding two more solutions: the stationary line and the self-translating grim reaper.
 This proof uses in an essential way X.J. Wang's dichotomy, which states that a convex ancient solution $\{\Gamma_t\}_{t\in(-\infty,0)}$ must either be entire (i.e. sweep out the whole plane, in the sense that $\cup_{t<0}\Omega_t=\R^2$, where $\Omega_t$ is the convex body bounded by $\Gamma_t$) or else lie in a  strip/slab region (the region bounded by two parallel  lines) \cite[Corollary 2.1]{Wa11}. Wang also proved that the only entire examples are the shrinking circles~\cite[Theorem~1.1]{Wa11}, thus reducing the classification question among solutions that lie in a slab.

In this paper, we aim to construct ancient compact convex solutions to the $\k^\a$ flow, for all $\a\in (1/2, 1]$, that lie in strip regions. These solutions can be thought of as the analogue to the Angenent ovals for the curve shortening flow.
Moreover, we prove that for $\a\in (\frac23, 1)$ the  Daskalopoulos, Hamilton and \v Se\v sum classification result extends to the $\k^\a$ flow. Finally, we show that any ancient non-compact convex solution to the $\k^\a$ flow, for all $\a\in (1/2, 1]$, must be a translating solution, therefore it is unique modulo rigid motions and parabolic rescalings by \cite{Urbas98}. We remark that this in fact is a corollary of the more involved result \cite[Theorem 1.1]{CCD18}.



\begin{theorem}\label{main thm} For any $\a\in (1/2, 1]$ there exists a convex compact ancient solution to the $\k^\a$ flow that lies between two parallel lines.
\end{theorem}

%
%
\begin{theorem}\label{thm2} 
For $\a\in(\frac23,1]$, any ancient compact convex solution to the $\k^\a$ flow must be the solution constructed in Theorem \ref{main thm} or the shrinking circle, modulo rigid motions and parabolic rescalings. 

For $\a\in (1/2, 1]$ any ancient non-compact convex solutions to the $\k^\a$ flow must be a translating solution.
\end{theorem}

Using the classification of translators by Urbas \cite{Urbas98}, we obtain the following corollary, which combined with the first part of Theorem \ref{thm2} provides a complete classification of ancient convex solutions to the $\k^\a$ flow for $\a\in (\frac23, 1]$.

\begin{corollary} 
For $\a\in(1/2,1]$, modulo rigid motions and parabolic rescalings, Urbas' translating solution over a strip \cite{Urbas98} and the straight line are the unique non-compact convex ancient solutions to the $\k^\a$ flow.
\end{corollary}

The dichotomy theorem of X.J. Wang has been extended to the $\k^\a$ flow, $\a\in (1/2, 1]$, by S. Chen \cite{Chen15}, who also showed that if the solution sweeps the whole $\R^2$ then it must be the shrinking circle. We use this result to reduce our analysis to solutions that lie in a slab. 

For $\a\in (1/2, 1)$, it does not seem possible to find an explicit solution as the Angenent oval. But for this range of $\a$'s,  it is known that, apart from the stationary line, and modulo rigid motions and parabolic translations, there is a unique translating solution to the $\k^\a$ flow and, in fact, this solution lies in a slab \cite{Urbas98}. This translating solution is essential in our construction of the solution described in Theorem \ref{main thm}

Our proof follows the ideas in \cite{BLT3}. We construct an ancient solution that lies in a strip by doubling (via reflection) compact pieces of the translating solution flowing them by the $\k^\a$ flow and taking a limit as the compact pieces become larger and larger. By analyzing the asymptotics of this solution, we are able to use Alexandrov reflection principle to show first that any such solution is reflection symmetric with respect to the mid-plane of the slab and second, for $\a\in (\frac23,1]$, that the solution is unique. Many of the techniques used in the construction as well as in the proof of uniqueness of this solution, have been used in \cite{BLT1, BLT3}. A major difficulty encountered for $\a<1$ is that the derivative of the enclosed area is no longer constant, which was a crucial ingredient in the previous works. 

The fact that we can prove uniqueness only for $\a\in (\frac23,1]$ comes from the better asymptotics we get in this range, which is due to the following reason.  
Urbas' unique convex translating solution is a graph over some bounded interval $I$, which, after a rotation, can be taken to lie on the $x$-axis. If we consider a timeslice that is contained in $I\times (0, +\infty)$ and let $R$ be the non-convex region of $I\times(0, +\infty)$ delimited by this timeslice, then the area of $R$ is finite if an only if $\a\in (\frac23, 1]$.  This area estimate allows us to obtain good enough asymptotics of the constructed solution in order to apply Alexandrov reflection and the maximum principle to prove uniqueness. We remark here that the same issue appears in the recent work of B. Choi, K. Choi and Daskalopoulos \cite{CCD20}, where they construct ancient solutions to the Gauss curvature flow under an initial finiteness of volume assumption. \\

The paper is structured as follows: In Section \ref{sec:AncientSol}, we gather facts about convex ancient solutions to the $\k^\a$ flow; In Section \ref{sec:Construction}, we construct families of old but not ancient solutions and show they satisfy estimates similar to the ones in Section \ref{sec:AncientSol} so that a subsequence of them converges to an ancient solution; Finally, in Section \ref{sec:Uniqueness}, we prove the  classification Theorem \ref{thm2}. 

\subsection*{Acknowledgements} This research originated at  the  workshop ``Women in Geometry 2" at the  Casa Matem\'atica Oaxaca (CMO)  from June 23 to June 28, 2019.  We would like to thank CMO-BIRS for creating the opportunity to start work on this problem through their support of the workshop. Theodora Bourni thanks Mat Langford and Toti Daskalopoulos for useful discussions on the subject.

\section{Ancient solutions of the $\k^\a$-flow}
\label{sec:AncientSol}

Let  $\{\Gamma_t\}_{t\in (-\infty, 0)}\subset \R^2$ be a convex ancient solution to the $\k^\a$-flow and consider its parametrization by its turning angle
\[
\gamma:\Theta\times(-\infty,0)\to \R^2, \quad \Theta\subset[0, 2\pi)\,.
\]
The turning angle $\theta$ of the solution is the angle made by the $x$-axis and its tangent vector with respect to a counterclockwise parametrization. 
Using spacetime translation and a space rotation if necessary, we assume that $\lim_{t\to 0}\Gamma_t\subset \{y\ge 0\}$ and that, at time $0$, the solution vanishes in the compact case whereas in the non-compact case it passes through the origin.  

Let  $\nu=\nu(\theta, t)$ and $\k(\theta, t)$ be the outward pointing unit normal and curvature of $\Gamma_t$ at $\gamma(\theta, t)$, respectively. Then
\[
\nu(\theta, t)=(\sin\th, -\cos\th)\,,
\]
and the evolution of $\k$ is given by 
\begin{equation}\label{kt}
\k_t=\k^2(\k^\a)_{\th\th}+\k^{\a+2}\,.
\end{equation}

A very useful feature of the $\k^\a$ flow is that it satisfies a differential Harnack inequality, a consequence of which is that the curvature on an ancient solution is non decreasing \cite{BA94, Chow91}:
\begin{equation}\label{Har}
\k_t(\theta, t)\ge 0\,.
\end{equation}
This holds for non-compact solutions as well, since the curvature is bounded at all timeslices. Moreover, it is know that the inequality is strict unless the solution moves by translation. As a corollary we obtain that the ``ends'' of an ancient solution are translators in the following sense. For any $\theta\in \Theta$ and any sequence of times $t_i\to -\infty$ the sequence of flows
$\Gamma_t^i=\Gamma_{t+t_i}-\gamma(\theta, t_i)$
converges, after passing to a subsequence,  locally uniformly in the smooth topology to a convex translating solution to the $\k^\a$ flow. For a proof of this standard compactness argument see for instance \cite[Lemma 5.2]{BLT1}. In fact, the curvature $\k^T(\theta, t)$ of the limiting translator satisfies 
\[
\k^T(\theta, t)=\lim_{j\to \infty}\k(\theta, t_{i_j}).
\]

\subsection*{The convex translating solution}
According to the theorem of Urbas stated below \cite{Urbas98} (see also \cite[Proposition 2.4]{CCD18}),  modulo translations, there is a unique convex translating solution of the $\k^\a$-flow moving with speed $1$ along the $e_2$ direction, for $\a\in(1/2,1]$, other than the straight line. We thus have that the limit $\Gamma^s_t=\Gamma_{t+s}-\gamma(\theta, s)$, as $s\to -\infty$, exists.

\begin{theorem}\label{Urbas}\cite{Urbas98}
Modulo translations, there exists a unique strictly convex curve $G_0$ that satisfies 
\[
\k^\a=-\langle \nu, e_2\rangle\,,\,\,\a\in (1/2,1]\,,
\]
where $\k$ and $\nu$ are the curvature and downward pointing unit normal to $G_0$. Moreover $G_0$ is a graph over a strip of width
\[
w_\a:=\int_{-\infty}^{+\infty}\frac{1}{(1+y^2)^{\frac12(3-\frac1\a)}}\, dy \,.
\]
Modulo translations, $G_t= G_0 +te_2$ is then the unique convex translating solution that moves with speed 1 along the $e_2$ direction.
\end{theorem}

\subsection*{Estimates for compact convex ancient solutions}
From now on, $\a\in(1/2,1)$ and $\{\Gamma_t\}_{t \in (-\infty, 0)}$ will denote a compact convex ancient solution to the $\k^\a$ flow which lies in the strip $\left[-\frac{w_\a}{2}\frac{w_\a}{2}\right]\times\R$ and in no smaller strip, where $w_\a$ is as in Theorem \ref{Urbas}. Using the parametrization by turning angle, we have
\[
\gamma:(-\pi, \pi]\times (-\infty, 0)\to \left(-\frac{w_\a}{2}, \frac{w_\a}{2}\right)\times\R\,.
\]
Since the solution lies in no smaller strip, it is not difficult to see that $\cup_{t\in (-\infty, 0)} \Gamma_t=(-\frac{w_\a}{2}, \frac{w_\a}{2})\times \R$, see for instance \cite[Lemma~5.1]{BLT1}. 
We will use the differential Harnack inequality, to show bounds on the curvature and width of our convex ancient solutions.

\begin{proposition}\label{lowerkbbound}
For any $\th\in(-\pi, \pi]$ and $t\in (-\infty, 0)$
\[
\k^\a(\th, t)\ge \pm \langle \nu(\th, t), e_2\rangle.
\]
\end{proposition}
\begin{proof}
We note first that the function $\k^\a\mp\langle\nu, e_2\rangle$ satisfies the Jacobi equation for the $\k^\a$ flow: 
\begin{equation}\label{Jaceqn}
u_t=\a\k^{\a+1}(u_{\th\th}+u)\,.
\end{equation}
The inequality is trivially true for $\theta=\pm \frac{\pi}{2}$. Recall that the flows $\Gamma^s_t=\Gamma_{t+s}-\gamma(\theta, s)$ converge as $s\to -\infty$ to a convex translator which lies in a strip of width $w_\a$, which by Theorem \ref{Urbas}, satisfies $\k^\a\mp\langle\nu, e_2\rangle\ge 0$.
The result then follows by the maximum principle.
\end{proof}

\begin{proposition}\label{X-est}
For any $t\in (-\infty, 0)$ and  any $ \th\in \pm[0, \pi]$

\[
\pm \left(\langle\gamma(\pm \tfrac\pi2, t), e_1\rangle-\langle\gamma(\theta, t),e_1\rangle\right)\le \pm \int_\th^{\pm\frac\pi2}\frac{\cos u}{|\cos u|^\frac{1}{\a}}\,du \,.
\]

\end{proposition}
\begin{proof} Recall that a convex curve $\gamma(\th)$ with curvature $\k(\theta)$ 
satisfies
\begin{equation}
\label{eq:DiffPosition}
  \gamma(\th_1) - \gamma(\th_0)=\left(\int_{\th_0}^{\th_1}\frac{\cos u}{\k(u)}du\,,\,\,\int_{\th_0}^{\th_1}\frac{\sin u}{\k(u)}du\right)\,.
\end{equation}
For any $\th\in[0, \frac{\pi}{2}]\cap \Theta$ we compute, using Proposition \ref{lowerkbbound} and the fact that $-\langle\nu,e_2\rangle=\cos \th$,
\[
\langle\gamma(\tfrac\pi2, t), e_1\rangle-\langle\gamma(\theta, t),e_1\rangle=\int_\th^{\pi/2}\frac{\cos u}{\k(u,t)}du\le \int_\th^{\pi/2}\frac{\cos u}{|\cos u|^{\frac1\alpha}}\,du\,.
\]
The other cases are proved similarly.
\end{proof}
As a consequence of this width estimate and our condition that the ancient solution $\Gamma_{t}$ lies in no strip smaller than $\left[-\frac{w_\a}{2}, \frac{w_\a}{2}\right]\times\R$,  we can determine scale of the limiting translating solution:
\begin{proposition}\label{edges}
For any $\theta\in (-\pi, \pi]\setminus\{\frac\pi2, -\frac{\pi}{2}\}$, the sequence of flows 
$\Gamma^s_t=\Gamma_{t+s}-\gamma(\theta, s)$
converges to a translating solution, which after a time and space translation, as well as a reflection about the $x$-axis in case $\theta\in (-\pi, \pi]\setminus(-\tfrac{\pi}{2}, \tfrac{\pi}{2})$, is given by $\{G_t\}_{t\in(-\infty, \infty)}$, as defined in Theorem~\ref{Urbas}.
\end{proposition}
\begin{proof} Given $\theta\in (-\pi, \pi]\setminus\{-\frac\pi2, \frac\pi2\}$, we have already seen that the sequence of flows 
$\Gamma^s_t=\Gamma_{t+s}-\gamma(\theta, s)$, as $s\to -\infty$, converges to a translator, which by Theorem \ref{Urbas}, and after a time and space translation and possibly a reflection about the $x$-axis, is given by $\{\lambda G_t\}_{t\in (-\infty, \infty)}$, with the scale satisfying $\lambda\le 1$ since the translator must be contained in a slab of width $w_\a$ (and it might be contained in a smaller one). We want to show here that the scale satisfies $\lambda =1$.

Let us assume that $\theta\in[0, \tfrac\pi2)$ (the other cases being treated similarly). Since the solution $\{\Gamma_t\}_{t\in (-\infty, 0)}$ is contained in a slab of width $w_\a$ and in no smaller slab, for any $\e>0$ there exists $t_\e<0$ such that 
\begin{equation}\label{widthest}
 \tfrac{w_\a}{2}-\langle\gamma(\tfrac\pi2, t) , e_1\rangle \le \e \,,\,\,\forall t\le t_\e\,.
\end{equation}
Note that, with the change of variable $y = \tan \theta$, 
\begin{equation}\label{wath}
\frac{w_\a}{2}=\int_{0}^{+\infty}\frac{1}{(1+y^2)^{\frac12(3-\frac1\a)}}\,dy =\int_0^\frac\pi2\cos\th^{1-\frac1\a}d\th\,,
\end{equation}
and thus we can choose $\theta\in (0, \tfrac\pi2)$ so that$\int_\th^\frac\pi2\cos\th^{1-\frac1\a}d\th<\e$. Then, by Proposition \ref{X-est} and \eqref{widthest}
\[
\frac{w_\a}{2}- \langle\gamma(\th, t) , e_1\rangle\le  \langle\gamma(\tfrac{\pi}{2}, t) , e_1\rangle- \langle\gamma(\th, t) , e_1\rangle+\e\le 2\e\,.
\]
Hence the limit translator $\{\lambda G_t\}_{t\in (-\infty, \infty)}$ cannot be contained in a slab smaller that $w_\a-4\e$, and since $\e$ was arbitrary we have $\lambda=1$ which yields the result.
\end{proof}

\begin{proposition}\label{area} Let $A(t)$ be the area of the region enclosed by $\Gamma_t$. Then
\[
A(t)\le 2w_\a(-t)\,.
\]
\end{proposition}
\begin{proof}
 Using Proposition \ref{lowerkbbound} and the fact that $\langle \nu, e_2\rangle=-\cos\th$, we have
\[
\begin{split}
-\frac{dA(t)}{dt}&=\int_{-\pi}^{\pi} \k^{\a-1}(\th)\,d\th\le4\int_0^\frac\pi2\cos\th^{1-\frac1\a}d\th\,.
\end{split}
\]
Hence, recalling \eqref{wath}, we have
\[
\begin{split}
-\frac{dA(t)}{dt}&\le 2w_\a\\
\end{split}
\]
and  integrating from $t$ to $0$,  yields the result. 
\end{proof}

\section{Constructing a compact ancient solution}
\label{sec:Construction}

We will use the translating solution $\{G_t\}_{t\in(-\infty, \infty)}$, as described in Theorem \ref{Urbas}, to construct a compact ancient solution. By the uniqueness of $G_0$,  after a space and time translation,  we can assume that $G_0$ passes through the origin and is reflection symmetric with respect to the $y$-axis, so $G_0\subset \{y\ge 0\}$. Let $\Gamma^R$ be the convex curve we obtain after taking the union of $G_{-R}\cap \{y\le0\}$ and its reflection along the $x$-axis, see Figure \ref{Figure}. This curve is not smooth, but it is convex and thus, a solution to the $\k^\a$ flow  ``flowing out'' of this curve exists, in the sense of the following theorem of Andrews.

\begin{theorem} \label{convexsln}\cite{BA98} 
For any convex curve $\Gamma_0$ bounding an open convex region in $\R^2$,
there exists a family of embeddings $\gamma: S^1\times (0, T)$ unique up to time-independent
reparametrization, which satisfy equation \eqref{ka-flow} and such that the image curves $\Gamma_t$
 converge to $\Gamma_0$ in Hausdorff distance as $t\to 0$. Moreover $\gamma\in C^\infty(S^1\times(0,T))$ and $\Gamma_t$ is strictly convex for all $t>0$.
\end{theorem}

\begin{figure}[htbp]
\begin{tikzpicture}[line cap=round,line join=round, scale=1.3]
  \pgfmathsetmacro{\a}{1.50}; 
  \pgfmathsetmacro{\am}{-1*\a}; 
  \pgfmathsetmacro{\p}{3.1416/2}; 
  \draw [->] (-2, 0) -- (2, 0) node[right] {$x$};

    \draw [line width=2pt,smooth,variable=\x,green] plot [samples=100, domain={\am-.06}:{\a+.06}] (\x,{ln(cos(\x r))-ln(cos((\a+.06) r))});
    
  
  \draw [line width=2pt, domain={\am+.09}:{\a-.09},smooth,variable=\x,teal] plot (\x,{ln(cos(\x r))-ln(cos(\a r))});
  \draw [line width=2pt, smooth, teal] ({\am+.09}, {ln(cos((\am + .09) r))-ln(cos(\a r))}) .. controls (\am+.03,0.46) and (\am+.03, 0.3) .. (\am+.03, 0);
  \draw [line width=2pt, smooth, teal] ({\a-.09}, {ln(cos((\am + .09) r))-ln(cos(\a r))}) .. controls (\a-.03,0.46) and (\a-.03, 0.3).. (\a-.03, 0);

  \begin{scope}[yscale=-1,xscale=1]
  
  
    \draw [line width=2pt,smooth,variable=\x,green] plot [samples=50, domain={\am-.06}:{\a+.06}] (\x,{ln(cos(\x r))-ln(cos((\a+.06) r))});
    
  \draw [line width=2pt, domain={\am+.09}:{\a-.09},smooth,variable=\x,teal] plot (\x,{ln(cos(\x r))-ln(cos(\a r))});
  \draw [line width=2pt, smooth, teal] ({\am+.09}, {ln(cos((\am + .09) r))-ln(cos(\a r))}) .. controls (\am+.03,0.46) and (\am+.03, 0.3) .. (\am+.03, 0);
  \draw [line width=2pt, smooth, teal] ({\a-.09}, {ln(cos((\am + .09) r))-ln(cos(\a r))}) .. controls (\a-.03,0.46) and (\a-.03, 0.3).. (\a-.03, 0);
    \end{scope}
    
            \draw [line width=1.5pt ,smooth,variable=\x,loosely dashed, black] plot [samples=100, domain={-1*(\p-.00005)}:{.01}] (\x,{-ln(cos(\x r))+ln(cos((\a+.06) r))});
        \begin{scope}[yscale=1,xscale=-1]
        \draw [line width=1.5pt ,smooth,variable=\x,loosely dashed, black] plot [samples=100, domain={-1*(\p-.00005)}:{.01}] (\x,{-ln(cos(\x r))+ln(cos((\a+.06) r))});
        \end{scope}
       \draw[{Stealth}-{Stealth}, line width=1pt] (-\p,5.5) -- (\p, 5.5) node [pos=.5, above] {$w_{\alpha}$};
   
  
  \draw[color=green] (-.7,3.7) node {$\Gamma^R$};
  \draw[color=teal] (-0.7,-1.6) node { $\Gamma^R_t$};
    \draw[color=black] (1.2,4.5) node {$G_{-R}$};
  \draw[{Stealth}-{Stealth}, line width=1pt] (0,0) -- (\a-.03, 0) node [pos=.5, above] {$h_R(t)$};
  \draw[{Stealth}-{Stealth}, line width=1pt] (0,0) -- (0, {-ln(cos(\a r))}) node [pos=.5, fill=white] {$\ell_R(t)$};
  \draw[{Stealth}-{Stealth}, line width=1pt] (0,0) -- (0, -{-ln(cos((\a+.06) r))}) node [pos=.7, fill=white] {$R$};

  \draw[color=black] (2.5, 4) node[anchor=west] {$\Gamma^R = \Gamma^R_{T_R}$ is the initial curve}; 
  \draw[color=black] (2.5, 3.5) node[anchor=west] {$\Gamma^R_t$ is $t$-slice of the solution};
  \draw[color=black] (2.5, 3.0) node[anchor=west] {$\ell_R(t) = \sup_\theta\langle\gamma^R(\theta, t), e_2\rangle$};
  \draw[color=black] (2.5, 2.5) node[anchor=west] {$h_R(t) =  \sup_\theta\langle\gamma^R(\theta, t), e_1\rangle$};
  \draw[color=black] (2.5, 2.0) node[anchor=west] {$\mathcal{A}^R$ is the area enclosed by $\Gamma^R$};
  \draw[color=black] (2.5, 1.5) node[anchor=west] {$A_R(t)$ is the area enclosed by $\Gamma^R(t)$};

\end{tikzpicture}
\caption{The approximate solutions $\{\Gamma^R_t\}_{t\in (T_R, 0)}$.}
\label{Figure}
\end{figure}
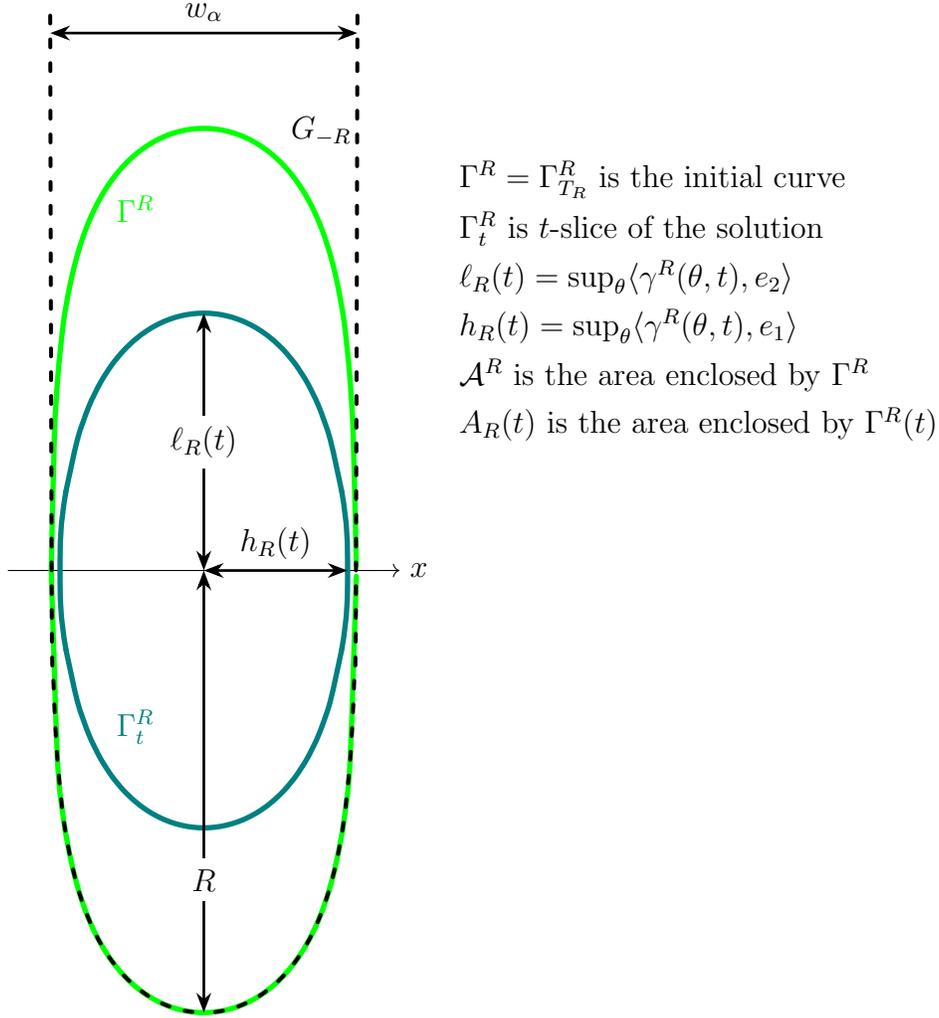

Recall that the solution shrinks to a point in finite time. We can therefore translate time so that the extinction time is $t=0$ for our solutions. With this convention in mind, we define $\{\Gamma^R_t\}_{t\in (T_R, 0)}$ with $T_R \in (-\infty, 0)$ to be a solution to \eqref{ka-flow} with initial data given by $\Gamma^R$ in the sense of Theorem \ref{convexsln} and let $\gamma^R:S^1\times (T_R, 0)\to \R^2$ be a parametrization  of $\{\Gamma^R_t\}_{t\in (T_R, 0)}$ by turning angle, see Figure \ref{Figure}. Moreover, by the uniqueness of the solution given in Theorem \ref{convexsln}, this solution is reflection symmetric with respect to both coordinate axes, which furthermore implies that the extinction point is the origin. 

We will construct an ancient solution  by taking a limit of such solutions as $R\to \infty$. For this, we will first prove that the curvature of and area enclosed by the solutions $\{\Gamma^R_t\}_{t\in (T_R, 0)}$ satisfy bounds similar to the estimates given for convex ancient solutions in Section \ref{sec:AncientSol}. 
\begin{proposition}\label{lowerkbboundR}
For any $R>0$, $t\in (T_R,0)$ and $\th\in S^1$,
\begin{equation*}\label{lowerkbboundR1}
\k^\a(\th,t)\ge \pm \langle\nu(\th,t), e_2\rangle\, .
\end{equation*}
\end{proposition}
\begin{proof} Because of Theorem \ref{convexsln}, the proof follows as in Proposition \ref{lowerkbbound}. Indeed, since
\[
\liminf_{t\to T_R}(\k^\a(\cdot, t)\pm \langle\nu(\cdot,t), e_2\rangle)\ge 0 
\]
and $\k^\a\mp\langle\nu, e_2\rangle$ satisfies \eqref{Jaceqn}, the result follows by the maximum principle.
\end{proof}
Next we show a certain monotonicity of the curvature throughout the flows. In particular, the result tells us that the maximum of $\k^{\alpha}$ occurs at $\theta=0$, so that we just need a bound on $\k(0, t)$ in order to get a uniform bound on the curvature $\k(\theta, t)$ for any $\theta$.
\begin{proposition}\label{k-monotone-t}
For any $R>0$ and $t\in (T_R,0)$, $\Gamma^R_t$ satisfies
\begin{equation}\label{k-monotone}
\begin{cases}
(\k^\a)_\theta\le 0\,,\text{ for }\theta\in (0, \pi/2)\cup(\pi, 3\pi/2)\\
(\k^\a)_\theta\ge 0\,,\text{ for }\theta\in (\pi/2, \pi)\cup(3\pi/2, 2\pi)\,.
\end{cases}
\end{equation}
\end{proposition}
\begin{proof}
Note first that on $\Gamma^R$ \eqref{k-monotone} is true by convexity and the translator equation.

Let $v=(\k^\a)_\th$. Then, recalling \eqref{kt}, we have
\[
\begin{split}
v_t=(\a+1)\k v ((\k^\a)_{\th\th} + \k^\a)+\a\k^{\a+1}(v_{\th\th}+v)\,.
\end{split}
\]
The strong maximum principle and the Hopf lemma then, applied to each of the four $\theta$- intervals implies the result.
\end{proof}
Next we prove some displacement and area estimates. For this we define
\[
\ell_R(t)=-\langle \gamma^R(0, t), e_2\rangle=\langle\gamma^R(\pi, t), e_2\rangle=\max_{\theta\in [0, 2\pi)}\langle\gamma^R(\th, t), e_2\rangle\,,
\]
\[
h_R(t)=\langle\gamma^R(\tfrac{\pi}{2}, t), e_1\rangle=-\langle\gamma^R(\tfrac{3\pi}{2}, t), e_1\rangle=\max_{\th\in[0, 2\pi)}\langle\gamma^R(\th, t), e_1\rangle\,,
\]
and $A_R(t)$ to be the area enclosed by $\Gamma^R_t$, see Figure \ref{Figure}. 

Note that, by the construction of the initial surface, we have 
$\ell_R(T_R)=R$ and $\lim_{R\to \infty} h_R(T_R)=\tfrac{w_{\alpha}}{2}$. We first show some more precise estimates on $h_R(T_R)$ as well as $A_R(T_R)$.

\begin{proposition}\label{initial data} There exist constants $R_\a>0$ and $C_\a>0$ such that for all $R\ge R_\a$
\begin{itemize}
\item[(i)] $h_R(T_R) \ge \frac{w_\a}{2}- C_\a R^\frac{1-2\a}{1-\a}$\,,
\item[(ii)] $A_R(T_R)\ge 2w_\a R- C_\a R^\frac{2-3\a}{1-\a}$\,.
\end{itemize}
\end{proposition}
\begin{proof}
For any $\theta_0\in(0, \tfrac{\pi}{2})$, using \eqref{wath}, we have
\begin{equation}\label{th0est}
\begin{split}
\int_{-\th_0}^{\th_0}\cos\th^{1-\frac1\a}d\th&=w_\a-2\int_{\tan\theta_0}^{+\infty}\frac{1}{(1+y^2)^{\frac12(3-\frac1\a)}}\,dy   \\
& \geq w_{\a} - 2\int_{\tan\theta_0}^{+\infty}y^{-3+\frac1\a}\,dy = w_\a-\frac{2\a}{2\a-1}(\tan\th_0)^{\frac1\a-2}\,.
\end{split}
\end{equation}
Let $\pm \theta_R$ be the turning angles at the two points of intersection of $\Gamma_R$ with the $x$-axis. Then by equation \eqref{eq:DiffPosition},
 \begin{equation}\label{thr}
 R=\int_0^{\th_R}\frac{\sin u}{\k(u,0)} du= \int_0^{\th_R}\frac{\sin u}{\cos u^\frac{1}{\a}} du=\frac{\a}{1-\a}(\cos\theta_R^{1-\frac{1}{\a}}-1)\,.
 \end{equation}
 Therefore
 \[
 \tan^2\theta_R=\cos\theta_R^{-2}-1=\left(\frac{1-\a}{\a} R+1\right)^\frac{2\a}{1-\a}-1\ge \left(\frac{1-\a}{2\a} R\right)^\frac{2\a}{1-\a}\,,
  \]
with the last inequality being true for all $R\ge R_\a$ for a constant $R_\a>0$.
 Now working as in \eqref{thr}, and using \eqref{th0est}, we obtain for $h_R(T_R)$
 \[
 \begin{split}
 h_R(T_R)=\int_0^{\th_R}\cos u^{1-\frac1\a} du&\ge \frac{w_\a}{2}-\frac{\a}{2\a-1}(\tan\theta_R)^\frac{1-2\a}{\a}\\
 &\ge \frac{w_\a}{2}-C_\a R^\frac{1-2\a}{1-\a}\,,
 \end{split}
 \]
 for some constant $C_\a$ and for all $R\ge R_\a$.
 
For the area estimate, we let $\mathcal{A}^R= A_R(T_R)$ be the area enclosed by $\Gamma^R$, so thanks to the symmetry with respect to the $x$-axis, $\frac12 \mathcal{A}^R$ is the area enclosed by $G_{-R}$ and the $x$-axis. Then,  following the proof of Proposition \ref{area} and estimating as above, we obtain
  \[
\frac12 \frac{d\mathcal{A}^R}{dR}=\int_{-\th_R}^{\th_R}\k^{\a-1}(\th)\,d\th= \int_{-\th_R}^{\th_R}\cos\th^{1-\frac1\a}d\th\ge w_\a-2 C_\a R^\frac{1-2\a}{1-\a}\,,
 \]
 for all $R\ge R_\a$.
Integrating from $R_\a$ to $R$ we obtain (ii).
\end{proof}

\begin{proposition}\label{t-Restimates} There exist constants $C_\a$ and $R_\a$ such that for all $R>R_\a$ and all $t\in (T_R,0)$ the following estimates hold.
\begin{itemize}
\item[(i)]  $-2w_a t\ge A_R(t)\ge -2w_\a t- C_\a(-t)^{2-2\a}$\,, 
\item[(ii)] $R\ge -T_R\ge R- C_\a(1+ R^\frac{2-3\a}{1-\a})$\,,
\item[(iii)]$\frac{w_\a}{2}\ge h_R(t)\ge \frac{w_\a}{2}- C_\a(-t)^{1-2\a}$\,,
\item[(iv)] $-t\le \ell_R(t)\le \min\{-t + C_\a\left(1+R^{\frac{2-3\a}{1-\a}}\right), -C_\a t\}$\,,
\item[(v)]$1\le \k_R^\a(0, t)\le C_\a\left(1+\frac{1}{-t}\right)$\,.
\end{itemize}
\end{proposition}
\begin{proof}
The proof of the first inequality in (i) follows the corresponding proof of Proposition \ref{area} for ancient solutions, using here  Proposition \ref{lowerkbboundR} instead of Proposition \ref{lowerkbbound}.

By Proposition \ref{lowerkbboundR}, we have
\[
-\ell_R'(t)=\k^\a_R(0, t)\ge -\langle \nu(0, t), e_2\rangle=1\,.
\]
Integrating from $t$ to $0$ yields the first inequality of (iv). As we shall use it later, we point out that integrating from $T_R$ to $t$ yields
\begin{equation}\label{elles}
\ell_R(t)\le R-(t-T_R)\,.
\end{equation}
Applying the first inequality of (iv) with $t=T_R$ yields the first inequality in (ii). The second inequality in (ii) is a consequence of the first inequality in (i) applied with $t=T_R$ and estimate (ii) of Proposition \ref{initial data}. 

The first inequality in (iii)  follows from the fact that $\Gamma^R$ is in the strip $\left(-\frac{w_{\alpha}}{2}, \frac{w_{\alpha}}{2}\right) \times \R$. To prove the second inequality, we use a technique from \cite[Lemma 4.4]{BLT1}.  Consider the circle centered on the negative $x$-axis and passing through $\gamma^R(\tfrac\pi2,t)$ and $\gamma^R(0,t)$. By Proposition \ref{k-monotone-t} and a simple maximum principle argument, it must be tangent to the curve at $\gamma^R(\tfrac{\pi}{2},t)$ from the inside, see \cite[Lemma 4.4]{BLT1} for details on this argument. We therefore have
\[
\k_R(\tfrac{\pi}{2}, t)\le \frac{2h_R(t)}{\ell_R^2(t)+h_R^2(t)}\,.
\]
We use this to compute
\[
-h_R'(t)=\k_R^\a(\tfrac{\pi}{2}, t)\le \left(\frac{2h(t)}{\ell^2(t)+h^2(t)}\right)^\a\le  \frac{2^\a h^\a(t)}{(-t)^{2\a}}\,,
\]
where in the last inequality we used the first inequality of (iv). Hence
\[
-\frac{1}{1-\a}(h_R^{1-\a}(t))'\le \frac{2^\a}{2\a-1}((-t)^{1-2\a})'
\]
and integrating from $T_R$ to $t$ we obtain
\begin{equation*}\label{hest}
h_R^{1-\a}(t)\ge h_R^{1-\a}(T_R) +\frac{2^\a(1-\a)}{2\a-1}\left((-T_R)^{1-2\a}-(-t)^{1-2\a}\right)\,
\end{equation*}
and using estimate (i) of Proposition \ref{initial data} and the first inequality in (ii) here yields the result. 

To prove the second inequality in (i), we compute using Proposition \ref{lowerkbboundR} and  \eqref{eq:DiffPosition} 
\[
-\frac{1}{4}\frac{dA_R(t)}{dt}=\int_{0}^{\frac\pi2} \k^{\a -1}\, d\theta \ge \int_{0}^{\frac\pi2} - \frac{\langle \nu,  e_2 \rangle}{\k(\theta, t)} \, d\theta = \int_{0}^\frac\pi2 \frac{\cos \theta}{\k(\theta,t)} d\theta=h_R(t).
\]
%
Integrating from $t$ to $0$, and using the second inequality of (iii), yields the result.

To prove the second inequality in (iv), we first note that  \eqref{elles} along with the second inequality in (ii) here implies that
\[
\ell_R(t)\le -t + C_\a\left(1+R^{\frac{2-3\a}{1-\a}}\right)\,.
\]
Furthermore, the first inequality in (i) and the second in (ii), along with  convexity, yield the uniform estimate $\ell_R(t)\le -C_\a t$.

Finally, the first inequality in (v) is straightforward by Proposition~\ref{lowerkbboundR}. The upper bound on $\k$ is a result of the differential Harnack inequality \cite{Chow91}, according to which, the quantity
\[
\k(\theta, t)(-t)^{\frac{\a}{\a+1}}
\]
is non decreasing in $t$. We thus have
\[
\begin{split}
\ell_R(t)&=\int_t^0\k^\a(0, s)ds\ge \k^\a(0, t)(-t)^\frac{\a^2}{\a+1}\int_0^t (-s)^{-\frac{\a^2}{\a+1}} ds\\
&=\frac{\a+1}{\a+1-\a^2}\k^\a(0, t)(-t)\,.
\end{split}
\]
which, along with the uniform bound  $\ell_R(t)\le -C_\a t$,  yields the result.
\end{proof}

The following is a more precise version of Theorem \ref{main thm}.

\begin{theorem}\label{conv}
For any $\a\in (1/2,1]$, there exists a compact convex ancient solution to the $\k^\a$ flow, $\{\Gamma_t\}_{t\in (-\infty, 0)}$, that sweeps out $(-\tfrac {w_\a}{2}, \tfrac{w_\a}{2})\times\R$. This solution is symmetric with respect to both coordinate axes and satisfies
\begin{equation}\label{h-t}
h(t)\ge\frac{w_{\alpha}}{2}-\frac{2}{(-t)^{2\a-1}}\,,
\end{equation}
where $h(t)=\langle\gamma(\frac\pi2, t), e_1\rangle$.
Moreover, if $\a\in (\tfrac23,1)$ then
\begin{equation}\label{ell-t}
\ell(t)\le -t+O(1)\,,
\end{equation}
as $t\to-\infty$, where $\ell(t)=\langle\gamma(\pi, t), e_2\rangle$.
\end{theorem}
\begin{proof}
Note that, by Proposition \ref{t-Restimates}, the diameter and the curvature of the solutions $\{\Gamma^R_t\}_{t\in (T_R, 0)}$ are bounded uniformly in $R$, since by Proposition \ref{k-monotone-t} the maximum of the curvature is at $\theta=0$. Recall also that, by Proposition \ref{t-Restimates}, $T_{R}\to -\infty$ as $R\to \infty$. Therefore there exists a sequence $R_j\to \infty$ such that the sequence of flows $\{\Gamma^{R_j}_t\}_{t\in (T_{R_j}, 0)}$ converges locally uniformly to a family of compact convex curves $\{\Gamma_t\}_{t\in (-\infty, 0)}$. The uniform lower bound on $h_R(t)$ of Proposition \ref{t-Restimates} shows that this family sweeps out  $(-\tfrac {w_\a}{2}, \tfrac{w_\a}{2})\times\R$ and that it satisfies \eqref{h-t}. In case $\a\in(\tfrac23,1)$, the estimate \eqref{ell-t} is a consequence of the estimate on $\ell_R(t)$ in Proposition \ref{t-Restimates} (iv), since $\frac{2-3\a}{1-\a}<0$.

Finally, this family is a weak solution of the $\k^\a$ flow, in the sense that it satisfies the avoidance principle with respect to any other smooth solution. This is easy to see since it is a limit of smooth flows. Therefore, by Theorem \ref{convexsln}, it is indeed a smooth solution.
\end{proof}

\section {Classification of convex ancient solutions}
\label{sec:Uniqueness}
In this section, we will classify all convex ancient solution  $\{\Gamma_t\}_{t\in(-\infty, 0)}$ to the $\k^\a$ flow for $\a\in(\frac23,1]$ which are contained in the strip $\Pi:=\{(x,y)
:|x|<w_\a/2\}$ and in no smaller strip. Moreover we will show, for any $\a\in(1/2,1]$, that the only non-compact ones are translating solutions. The proof is very similar to the corresponding proof for the curve shortening flow case ($\a=1$) given in \cite[Theorem 3.3]{BLT3}.
We first show that any compact convex ancient solution that lies in a strip is reflection symmetric with respect to the mid-line of the strip, which is proven by using 
the Alexandrov reflection principle \cite{Chow97,ChGu01}. The proof follows \cite[Lemma~2.4]{BLT3} (cf. \cite{BrLo}). 

The non-compact case can be in fact dealt with in a much simpler way by looking at the asymptotic translators at $+\infty$ and $-\infty$. This idea follows the proof of \cite[Theorem 1.2]{CCD20}, by using here Proposition \ref{edges}. We present this result first.

\begin{theorem}\label{thm2II}
 Any convex, non-compact ancient solution to the $\k^\a$ flow, $\a\in (1/2,1]$, must be a translating solution.
\end{theorem}
\begin{proof}
Since the solution is convex, it is a graph over some interval $I$ (bounded or unbounded) which, without loss of generality, is an interval of the $x$-axis. Moreover, by \cite{CDK19}, this solution must in fact be eternal and we denote it by $\{\Gamma_t\}_{t\in (-\infty,+\infty)}$.  As we have already discussed, $\Gamma_t- \gamma(e_2,t)$ converges, as $t\to-\infty$, to a translator $T$. $T$ is either a straight line or it is defined between two parallel lines \cite{Urbas98}.  B. Choi, K. Choi and Daskalopoulos in \cite{CCD18} observed that, by the result in \cite{CDK19}, a convex complete graph over an open interval $I$
(either bounded or unbounded), under the flow, remains a convex complete graph over the same interval $I$. Hence, the solution is either entire or lies between two parallel lines.  $\Gamma_t- \gamma(e_2,t)$ converges, as $t\to+\infty$ to a translator $T_0$. In the case the solution is between two parallel lines, $T_0=T$ by \cite[Theorem 1.1]{CCD18}.  In case the solution is entire, then so is $T_0$ (which can be easily verified, for example by using translating solutions as barriers) and thus again $T=T_0$. Hence in both cases $T=T_0$ and by the rigidity case of the Harnack inequality we obtain the result.
\end{proof}

\begin{lemma}\label{lem:refl} 
Let $\{\Gamma_t\}_{t\in(-\infty, 0)}$ be a strictly convex ancient solution to the $\k^\a$ flow, $\a\in(1/2,1]$, which is contained in the strip $\Pi:=\{(x,y)
:|x|<w_\a/2\}$ and in no smaller strip. Then $\Gamma_t$ is reflection symmetric about the $y$-axis for all $t<0$.
\end{lemma}
\begin{proof}
Set
\[
H_\beta:=\{(x,y)\in \R^2:x<\beta\}
\]
and denote by $R_\beta$ the reflection about $\partial H_\beta$. By Proposition \ref{edges},  the `tips' (or tip if the solution is non compact) of this solution converge to the unique translator as defined in Theorem \ref{Urbas}, which is reflection symmetric. Therefore, by the convexity of $\{\Gamma_t\}_{t\in(-\infty,0)}$, given any $\beta\in (0, \frac{w_\a}{2})$, there exists a time $t_\beta$ such that
\[
(R_\beta\cdot\Gamma_t)\cap (\Gamma_t\cap H_\beta)=\emptyset
\]
for all $t<t_\beta$ (for more details of this fact, see \cite[Claim 6.2.1]{BLT1}). By the strong maximum principle and the boundary point lemma for strictly parabolic equations  this is true for all $t<0$, see \cite[Theorem 2.2]{Chow97}. Taking $\beta\to 0$ implies that $R_0\cdot \Gamma_t$ lies to the left of $\Gamma_t\cap H_0$. The result now follows by repeating the argument with $H_\beta:=\{(x,y)\in \R^2:x>-\beta\}$.
\end{proof}

\begin{theorem}\label{unique}
Modulo translations, there exists only one strictly convex ancient solution $\{\Gamma_t\}_{t\in(-\infty, 0)}$ to the $\k^\a$ flow $\a\in (\tfrac23, 1]$ contained in the strip $\Pi:=\{(x,y)
:|x|<\tfrac{w_\a}{2}\}$ and in no smaller strip.
\end{theorem}
\begin{proof}
After a time and space translation, we can assume that the solution gets extinct at the origin at time $t=0$. Let 
\[
L(t)=-\langle \gamma(0, t), e_2\rangle+\langle \gamma(\pi, t), e_2\rangle\,.\]
By  Proposition \ref{lowerkbbound} we have $\k^\a(0, t)\ge 1$ and $\k^\a(\pi,t)\ge 1$ thus
\[
\frac{d}{dt}\left(L(t)+2t\right)\le 0\,.
\]
Therefore the limit
\[
L:=\lim_{t\to-\infty}(L(t)+2t)
\]
exists in $[0,\infty]$.

This quantity, for the particular compact solution we constructed in Theorem \ref{conv}, satisfies, by  \eqref{ell-t},
\[
L_0:=\lim_{t\to -\infty}(L_0(t)+2t)<\infty
\]
where $L_0(t)=2\ell(t)$, with $\ell(t)$ as in Theorem \ref{conv}. 
We next claim that for any solution, 
\begin{equation}\label{claim}
L=L_0\,.
\end{equation}
Suppose, contrary to the claim, that $L>L_0$ (the case $L<L_0$ is ruled out similarly). Then we can find $t_0$ such that 
\begin{equation}\label{eq:ellt0}
L(t)>L_0(t) \,\,\text{ for all }t<t_0\,.
\end{equation}
Let $\{\mathrm{\Gamma^0}_t\}_{t\in(-\infty,0)}$ be the particular solution constructed in Theorem~\ref{conv}. 
Define the halfspaces $H_\beta$ and the reflection $R_\beta$ about $\partial H_\beta$ as in Lemma \ref{lem:refl}. Given any $\beta\in(0,\frac{w_\a}{2})$, set
\[
\wt{\mathrm{\Gamma^0}}_t=(R_\beta\cdot \mathrm{\Gamma^0}_t)\cap \{(x,y)\in \R^2:x<0\}\,,
\]
 and
\[
\wt \Gamma_t=\Gamma_t\cap \{(x,y)\in \R^2:x<0\}\,,
\]
and let $(-r^0_t, r^0_t)=\partial\wt{\mathrm{\Gamma^0}}_t$ and $(r^-_t, r^+_t)=\partial\wt\Gamma_t$.
Then, by \eqref{eq:ellt0} and since $L_0$ is finite, there exists $c>0$ such that for all $t<t_0$, there exists $c_t\in (-c, c)$, such that $(-r^0_t+c_t, r^0_t+c_t)\subset (r^-_t, r^+_t)$. Moreover, by convexity and the convergence of the tips given in Proposition \ref{edges}, there exists $t_\beta<t_0$ depending on $\beta$ such that $(\wt{\mathrm{\Gamma^0}}_t+c_t)\cap\wt\Gamma_t=\emptyset$ for all $t<t_\beta$, with $c_t$ as above. It then follows by the strong maximum principle that $(\wt{\mathrm{\Gamma^0}}_s+c_t)\cap\wt\Gamma_s=\emptyset$ for all $t<t_\beta$ and $t<s<t_0$. Letting $\beta\searrow 0$, we find that there exists a constant $c_0$ such that  $\wt \Gamma_t\cap(\mathrm{\Gamma^0}_t+c_0)=\emptyset$ for all $t<t_0$. By Theprem \ref{lem:refl} we thus obtain that $\Gamma_t$ contains $\Gamma^0_t+c_0$ for all $t<t_0$, contradicting the fact that both curves reach the origin at time $t=0$. This finishes the proof of \eqref{claim}.

Now consider, for any $\t>0$, the solution $\{\Gamma_t^\tau\}_{t\in(-\infty,0)}$ defined by $\Gamma_t^\t=\Gamma_{t+\tau}$. Since
\[
L_\t:=\lim_{t\to-\infty}(\ell(t+\t)+t)> L= L_0\,,
\]
we may argue as above to conclude that $\Gamma_t^\t$ lies outside $\mathrm{\Gamma^0}_t +c_\tau$ for some constant $c_\tau$ and for all $t<0$. Taking $\tau\to 0$, we find that $\Gamma_t$ lies outside $\mathrm{\Gamma^0}_t$ for all $t<0$. Since the two curves reach the origin at time zero, they intersect for all $t<0$ by the avoidance principle. The strong maximum principle then implies that the two coincide for all $t$.
\end{proof}

\begin{proof}[Proof of Theorem \ref{thm2}]
The non-compact case has been shown in Theorem \ref{thm2II}. For the compact case and $\a\in (\frac23,1]$,
we have shown uniqueness as long as our solution is restricted in a slab. The theorem now follows by a result of Chen \cite{Chen15}, which states that if the solution is  not defined in a slab, then it must be the shrinking circle. 
\end{proof}

\bibliographystyle{acm}  
\bibliography{bibliography.bib}

\end{document}